\documentclass[11pt]{article}

\usepackage{amssymb, amsmath, amsthm, graphicx}
\usepackage[left=1in,top=1in,right=1in]{geometry}

\date{}

\theoremstyle{plain}
      \newtheorem{theorem}{Theorem}[section]
      \newtheorem{lemma}[theorem]{Lemma}

      \newtheorem{corollary}[theorem]{Corollary}
      
      \newtheorem{conjecture}[theorem]{Conjecture}
\theoremstyle{definition}

\theoremstyle{remark}

\def\VC{\mbox{\rm VC-dim}}
\def\LS{\mbox{\rm LS-dim}}

\title{Sunflowers in set systems of bounded dimension}
\author{Jacob Fox\thanks{Stanford University, Stanford, CA. Supported by a Packard Fellowship and by NSF award DMS-1855635. Email: {\tt jacobfox@stanford.edu.}} \and J\'anos Pach\thanks{R\'enyi Institute, Budapest and MIPT, Moscow. Supported by NKFIH grants K-176529, KKP-133864, Austrian Science Fund Z 342-N31, Ministry of Education and Science of the Russian Federation MegaGrant No. 075-15-2019-1926, ERC Advanced Grant ``GeoScape.'' Email:
{\tt pach@cims.nyu.edu}.}\and  Andrew Suk\thanks{Department of Mathematics, University of California at San Diego, La Jolla, CA, 92093 USA. Supported an NSF CAREER award, NSF award DMS-1952786, and an Alfred Sloan Fellowship. Email: {\tt asuk@ucsd.edu}.} }

\begin{document}

\maketitle

\begin{abstract}
Given a family $\mathcal F$ of $k$-element sets, $S_1,\ldots,S_r\in\mathcal F$ form an {\em $r$-sunflower} if $S_i \cap S_j =S_{i'} \cap S_{j'}$ for all $i \neq j$ and $i' \neq j'$. According to a famous conjecture of Erd\H os and Rado (1960), there is a constant $c=c(r)$ such that if $|\mathcal F|\ge c^k$, then $\mathcal F$ contains an $r$-sunflower.

We come close to proving this conjecture for families of bounded {\em Vapnik-Chervonenkis dimension}, $\VC(\mathcal F)\le d$. In this case, we show that $r$-sunflowers exist under the slightly stronger assumption $|\mathcal F|\ge2^{10k(dr)^{2\log^{*} k}}$. Here, $\log^*$ denotes the iterated logarithm function.

We also verify the Erd\H os-Rado conjecture for families $\mathcal F$ of bounded {\em Littlestone dimension} and for some geometrically defined set systems.
\end{abstract}

\section{Introduction}

An {\em $r$-sunflower} is a collection of $r$ sets whose pairwise intersections are the same. That is, $r$ distinct sets $S_1,\ldots,S_r$ form an $r$-sunflower if $S_i \cap S_j = S_{i'} \cap S_{j'}$ for all $i \not = j$ and $i' \not = j'$.  The term was coined by Deza and Frankl~\cite{DF}. For brevity, a $k$-element set is called a {\em $k$-set}.

Let $f_r(k)$ be the minimum positive integer $m$ such that every family of $k$-sets whose size is at least $m$ contains $r$ members that form an $r$-sunflower. Erd\H{o}s and Rado \cite{ErRa} proved that $f_r(k) \leq k!(r-1)^k$. The Erd\H{o}s-Rado ``sunflower conjecture'' states that there is a constant $C=C(r)$ depending only on $r$ such that $f_r(k) \leq C^k$.  Over the years, some small improvements have been made on the upper bound $k!(r-1)^k$, see \cite{AbHS, Ko}. Very recently, a breakthrough has been achieved by Alweiss, Lovett, Wu, and Zhang \cite{ALWZ}, who proved that \[f_r(k) \leq (cr^3\log k\log\log k)^k,\]
where $c$ is an absolute constant. For an alternative proof of this result, using Shannon capacities, see \cite{R}. Some weaker versions of the conjecture are discussed in \cite{AH, ErSz, NS}.
\smallskip

The aim of this note is to study the Erd\H os-Rado sunflower conjecture for families of bounded dimension.  Apart from set systems realized in low-dimensional Euclidean spaces, we consider two additional notions of dimension:  the {\em Vapnik-Chervonenkis dimension} (in short, VC-dimension) and the {\em Littlestone dimension} (LS-dimension), introduced in \cite{VC} and \cite{Little}, respectively. Both are important combinatorial parameters that measure the complexity of graphs and hypergraphs, and play important roles in statistics, algebraic geometry, PAC learning, and in model theory.
There is a growing body of results in extremal combinatorics and Ramsey theory which give much better bounds or stronger conclusions under the additional assumption of bounded dimension (see \cite{FPS1,FPS2}).
\smallskip

Given a family of sets $\mathcal{F}$ with ground set $V$,  the {\em VC-dimension} of $\mathcal{F}$, denoted by $\VC(\mathcal F)$, is the maximum $d$ for which there exists a $d$-element set $S\subset V$ such that for every subset $B\subset S$, one can find a member $A\in \mathcal{F}$ with $A\cap S=B$. In this case, we say that $S$ is {\em shattered} by $\mathcal F$.

Let $f^d_{r}(k)$ denote the least positive integer $m$ such that every family $\mathcal F$ of $k$-sets with $|\mathcal F|\ge m$ and $\VC(\mathcal F)\le d$ contains an $r$-sunflower. Clearly, we have $f^d_{r}(k) \leq f_r(k)$, and the Erd\H os-Rado sunflower conjecture implies the following weaker conjecture.

\begin{conjecture}\label{conje}
For $d\geq 1$ and $r\geq 3$, there is a constant $C =C(d,r)$ such that $f^d_r(k) \leq C^k$.
\end{conjecture}

It is not difficult to see that, even for $d=1$, the function $f^1_r(k)$ grows at least exponentially in $k$. More precisely, we have $f^1_r(k) > (r-1)^{k-1}$. Indeed, consider a rooted complete $(r-1)$-ary tree $T$ with the root on level $0$ and with $(r-1)^{k-1}$ leaves on level $k-1$. Let $\mathcal{F}$ be the family of $k$-sets consisting of the vertex sets of the root-to-leaf paths in $T$. Obviously, $\mathcal{F}$ does not contain any $r$-sunflower, and its VC-dimension is at most 1. 

More generally, we have the recursive lower bound \[f^d_r(k_1+k_2) > (f^d_r(k_1)-1)(f^d_r(k_2)-1).\] Indeed, for $i=1,2$, let $\mathcal{F}_i$ be a family of  $k_i$-sets of size $f^d_r(k_i)-1$ with VC-dimension at most $d$ and without any $r$-sunflower. For each set $S$ in $\mathcal{F}_1$, make a new copy of $\mathcal{F}_2$ and add $S$ to each set in $\mathcal{F}_2$. 
The ground set of copies of $\mathcal{F}_2$ are pairwise disjoint for distinct sets of $\mathcal{F}_1$. The resulting set system $\mathcal{F}$ is $(k_1+k_2)$-uniform with size $(f^d_r(k_1)-1)(f^d_r(k_2)-1)$, VC-dimension at most $d$, and has no $r$-sunflower. This implies that if $f_r(k')>C^{k'}+1$ for some $k'$ and $C$, then there is $d$ depending on $k'$ such that for all sufficiently large $k$, $f^d_r(k)>C^k$. Thus, any exponential lower bound for the classical sunflower problem (with unbounded VC-dimension) can be achieved by a construction with bounded (but sufficiently large) VC-dimension. 
\smallskip

Using a result of Ding, Seymour, and Winkler \cite{DSW}, we settle Conjecture~\ref{conje} for families of $k$-sets with VC-dimension $d = 1$.

\begin{theorem}\label{VC1}
For integers $r\geq 3$ and $k\ge 1$, every family of $k$-sets with VC-dimension $d=1$ and cardinality at least $r^{10k}$  has an $r$-sunflower.  That is, we have
\[f^1_r(k) \leq r^{10k}.\]
\end{theorem}

Let $\log^*k$ denote the {\em iterated logarithm} of $k$, {\em i.e.}, the minimum $i$ for which the $i$ times iterated logarithm of $k$ satisfies $\log^{(i)} k \leq 2$. All logarithms used in this note are of base 2.

For $d\ge 2$, our upper bound on $f_r^d(k)$ is not far from the one stated in Conjecture~\ref{conje}.

\begin{theorem}\label{main}
For integers $d, k, r \geq 2$, every family of $k$-sets with VC-dimension at most $d$ and cardinality at least $2^{10k(dr)^{2\log^* k}}$  has an $r$-sunflower. In notation,
\[f^d_r(k) \leq 2^{10k(dr)^{2\log^* k}}.\]
\end{theorem}

The {\em Littlestone dimension} of $\mathcal F\subseteq 2^V$ is defined as follows. Consider a rooted complete binary tree $T_d$, with the root at level $0$ and with $2^{d}$ leaves at the last level. Let the leaves of $T_d$ be labeled by sets in $\mathcal F$, and all other vertices by elements of $V$. We say that $T_d$ is {\em shattered} by $\mathcal F$ if for every root-to-leaf path with labels $v_0, v_1,\ldots, v_{d-1}, F,$ we have $v_i\in F$ if and only if the $(i+1)$st vertex along the path is the left-child of $v_i$, for all $0\le i<d$. The Littlestone dimension of $\mathcal F$, denoted by $\LS(\mathcal F)$, is the largest $d$ for which there is a labeling of $T_d$ which is shattered by $\mathcal F$.
\smallskip

Obviously, we have $\VC(\mathcal F)\le \LS(\mathcal F)$, because if the $S=\{s_0,\ldots,s_{d-1}\}\subseteq V$ is shattered by $\mathcal F$, then the labeling of $T_d$ in which all vertices at level $i$ are labeled by $s_i$, $0\le s_i<d$, and the leaves by the corresponding sets in $\mathcal F$ with the appropriate intersection with $S$, is also shattered by $\mathcal F$.

Let $h^d_r(k)$ denote the least positive integer $m$ such that every family $\mathcal F$ of $k$-sets with $|\mathcal F|\ge m$ and $\LS(\mathcal F)\le d$ contains an $r$-sunflower.
Since the Littlestone dimension of a set system is at least as large as its VC-dimension, we have
\[h^d_r(k)\le f^d_r(k)\le f_r(k).\]
It turns out that $h^d_r(k)$, as a function of $k$, grows much more slowly than $f^d_r(k)$. Its growth rate is only polynomial in $k$, albeit the degree of this polynomial depends on $d$.

\begin{theorem}\label{littlestone}
For positive integers $d,r,k$, every family of $k$-sets with LS-dimension at most $d$ and cardinality at least $(rk)^{d}$  has an $r$-sunflower. Using our notation, we have
\[h^d_r(k)\le (rk)^{d}.\]

\noindent On the other hand, for integers $d,r \geq 3$, and $k\geq 4d$, we have \[h^{d}_r(k) \geq (rk/d)^{d-o(d)},\]
where the $o(d)$ term goes to $0$ as $d \to \infty$. 
\end{theorem}

For several geometrically defined set systems, one can verify the sunflower conjecture by exploring the special properties of the underlying configurations.
\smallskip

A collection $\mathcal{D}$ of Jordan regions in the plane is called a family of \emph{pseudo-disks} if the boundaries of any two members in $\mathcal{D}$ intersect in at most two points. For simplicity, we will assume that $\mathcal{D}$ is in \emph{general position}, that is, no point lies on the boundary of three regions and no two regions are tangent. It is well known that the VC-dimension of the set system obtained by restricting $\mathcal D$ to $V$ is at most $3$ (see \cite{BHP}) and, hence, Theorem~\ref{main} applies. However, in this case, we can verify the sunflower conjecture.

\begin{theorem}\label{disks}
Let $V$ be a planar point set and let $\mathcal{D} = \{D_1,\ldots,D_N\}$ be a family of pseudo-disks such that the size of every set $S_i = D_i\cap V$ is equal to $k$.  If $N \geq (500  + r)^{900k}$, where $r > 2$, then there are $r$ distinct sets $S_{i_1},\ldots,S_{i_r}$ that form an $r$-sunflower.
\end{theorem}

Our paper is organized as follows. Sections~\ref{sec2} and \ref{var} contain the proofs of Theorems~\ref{VC1} and \ref{main}, respectively. Theorem~\ref{littlestone} about set systems of bounded Littlestone dimension is established in Section~\ref{sec4}. Section~\ref{sec5} is devoted to low-dimensional geometric instances of the sunflower conjecture, while the last section contains some concluding remarks.
\smallskip

For the clarity of presentation, throughout this paper we make no attempt to optimize the absolute constants occurring in the statements.

\section{VC-dimension 1--Proof of Theorem~\ref{VC1}}\label{sec2}

Given a family $\mathcal{F}$ of subsets of a ground set $V$, as usual, let $\nu(\mathcal{F})$ denote the {\em packing number} of $\mathcal{F}$, {\em i.e.,} the maximum number of pairwise disjoint members of $\mathcal F$. Also, let $\tau(\mathcal{F})$ be the {\em transversal number} of $\mathcal{F}$, {\em i.e.}, the minimum number of elements that can be selected from $V$ such that every member of $\mathcal F$ contains at least one of them. Finally, let $\lambda(\mathcal{F})$ denote the maximum integer $l$ such that there are $l$ sets $S_1,\ldots, S_l \in \mathcal{F}$ with the property that for any $1 \leq i < j \leq l$, there is $v=v_{ij} \in S_i\cap S_j$ such that $v \not\in S_{t}$ for $t \in [m]\setminus\{i,j\}$.  It is easy to verify that $\lambda(\mathcal{F})$ is at least as large as the VC-dimension of the set system (hypergraph) $\mathcal{F}^*$ {\em dual} to $\mathcal{F}$.
\smallskip

We need the following result of Ding, Seymour, and Winkler \cite{DSW} which bounds the transversal number of $\mathcal{F}$ in terms of its packing number and $\lambda(\mathcal{F})$.

\begin{lemma}[Ding, Seymour, Winkler]\label{stab}
Let $\mathcal{F}$ be a set system with ground set $V$, and let $\nu(\mathcal{F}) = \nu, \tau(\mathcal{F}) = \tau$ and $\lambda(\mathcal{F}) = \lambda$.   Then we have

\[\tau \leq 11\lambda^2(\lambda + \nu + 3)\binom{\lambda + \nu}{\lambda}^2.\]

\end{lemma}

Notice that $\VC(\mathcal{F})=1$ implies that $\lambda(\mathcal{F}) \leq 3$.  Hence, Theorem~\ref{VC1} is an immediate corollary to the following result.

\begin{theorem}\label{lambda}
Let $r \geq 3$ and let $\mathcal{F}$ be a family of $k$-sets with $\lambda(\mathcal{F})=\lambda$ which does not contain an $r$-sunflower.  Then we have $|\mathcal{F}| \leq (\lambda + r)^{6\lambda k}$.

\end{theorem}

\begin{proof}
We proceed by induction on $k$. The base case $k=1$ follows from the trivial bound $|\mathcal{F}| \leq r-1$. The induction hypothesis is that the bound holds for families of $(k-1)$-sets. For the inductive step, let $\mathcal{F}\subseteq 2^V$ be a family of $k$-sets with no $r$-sunflower. In particular, $\mathcal F$ has no $r$ disjoint members, so that $\nu(\mathcal{F}) < r$.  By Lemma \ref{stab},
\begin{equation}\nonumber
\begin{array}{lcl}
\tau(\mathcal{F}) &\leq&  11\lambda^2(\lambda + r + 3)\binom{\lambda + r}{\lambda}^2\leq 11\lambda^2(\lambda + r + 3)(\lambda+r)^{2\lambda}(\lambda!)^{-2}\\
&\leq& 11(\lambda + r + 3)(\lambda+r)^{2\lambda} \leq 20(\lambda+r)^{2\lambda+1}.
\end{array}
\end{equation}
\noindent Therefore, there is $v\in V$ incident to at least
$|\mathcal{F}|/\tau(\mathcal{F}) \geq |\mathcal{F}|/\left(20(\lambda + r)^{2\lambda + 1}\right)$ members of $\mathcal{F}$.

Let $\mathcal{F}' = \{S\setminus \{v\}: S \in \mathcal{F}, v \in S\}$. Then we have  $|\mathcal{F}'| \geq |\mathcal{F}|/\left(20(\lambda+r)^{2\lambda+1}\right)$,
$\lambda(\mathcal{F}') \leq \lambda(\mathcal{F})$, and $\mathcal{F}' $ does not contain any $r$-sunflower. By the induction hypothesis, we have $|\mathcal{F}'| \leq (\lambda + r)^{6\lambda(k -1)}$. Thus, we obtain
\[|\mathcal{F}| \leq 20(\lambda + r)^{2\lambda + 1}|\mathcal{F}'| \leq 20(\lambda + r)^{2\lambda + 1} (\lambda + r)^{6\lambda(k -1)} \leq  (\lambda + r)^{6\lambda k },\] as required.\end{proof}

\section{Bounded VC-dimension--Proof of Theorem \ref{main}}\label{var}

In this section, we prove Theorem \ref{main}, which is the main result of this paper. We need the following lemma due to Sauer \cite{Sa}, Shelah \cite{Sh}, Perles, and, in a slightly weaker form, to Vapnik and Chervonenkis \cite{VC}. See also \cite{pa, F}.

\begin{lemma}[Sauer, Shelah, Perles]\label{ss}
Let $\mathcal{F}$ be a set system with ground set $V$ and VC-dimension at most $d$.  Then we have $|\mathcal{F}| \leq \sum_{i = 0}^d \binom{|V|}{i}$.
\end{lemma}

Before turning to the proof, we need to discuss some closely related variants of the sunflower problem.
\smallskip

First, we could ask the same question for {\em multifamilies} of sets, that is, for collections of not necessarily distinct sets. Let $g_r(k)$ be the minimum positive integer $m$ such that every multifamily of $k$-sets of size $m$ contains an $r$-sunflower. It is an easy exercise to prove that $g_r(k)=(r-1)f_r(k)+1$.

Analogously, for any $d\ge 1$, let $g^d_r(k)$ be the minimum positive integer $m$ such that every multifamily of $k$-sets of size $m$ with {\em VC-dimension at most $d$} contains an $r$-sunflower. We similarly have $g^d_r(k)=(r-1)f^d_r(k)+1$.
\smallskip

To obtain upper bounds for $g^d_r(k)$ and $f^d_r(k)$, we define the following related function.  Let $\alpha^d_r(k)$ denote the maximum $\alpha$ such that for every nonempty multifamily $\mathcal{F}$ of $k$-sets with VC-dimension at most $d$, if we select $r$ members uniformly at random from $\mathcal{F}$ with replacement, the probability that they have pairwise equal intersections is at least $\alpha$.

Next, notice that the value of $f_r(k)$ remains the same if we change the definition from families of $k$-sets to families of sets with {\em at most} $k$ elements. Indeed, this can be achieved by adding distinct ``dummy'' vertices to each set of size smaller than $k$ so that it will have size exactly $k$. The same holds for the functions $f^d_r(k), g_r(k),$ $g^d_r(k)$, and $\alpha^d_r(k)$ because adding dummy vertices does not affect the VC-dimension of the family.
\smallskip

Considering a family of VC-dimension $d$ which consists of $f^d_r(k)-1$ sets of size $k$ and contains no $r$-sunflower, we immediately obtain the following upper bound on $\alpha^d_r(k)$ as the $r$-tuples of sets from the family that have pairwise equal intersections are those that consist of the same set $r$ times. 
\begin{equation}\label{upper}
 \alpha^d_r(k) \leq (f^d_r(k)-1)^{1-r}.
\end{equation}

The following lemma implies that this bound on $\alpha^d_r(k)$ is tight within a factor $er^{r-1}$.

\begin{lemma}\label{small1}
For integers $d,k,r \geq 2$ we have
\[\alpha^d_r(k) \geq g^d_r(k)^{1-r}/e.\]
\end{lemma}

\begin{proof}

Let $S^d_r(m,k)$ denote the minimum possible number of $r$-sunflowers in a multifamily $\mathcal{F}$ of at most $k$-element sets with cardinality $m$ and VC-dimension at most $d$. From the definition, if $m<g^d_r(k)$, then $S^d_r(m,k)=0$, while if $m \geq g^d_r(k)$, then $S^d_r(m,k) \geq 1$.

Our argument is based on the proof technique used to obtain the ``crossing lemma'' \cite{AjCNS}, see also \cite{SmS}. The idea is to use an averaging (or, equivalently, probabilistic) argument to amplify a weak bound to a better bound. By deleting one set from each $r$-sunflower, we get the trivial bound $S^d_r(m,k) \geq m-g^d_r(k)+1$. For $M \geq m$, by averaging over all subfamilies of size $m$, we obtain 
\[S^d_r(M,k) \geq S^d_r(m,k)\binom{M}{r}/\binom{m}{r}.\] In particular, $S^d_r(m,k)/\binom{m}{r}$ is a monotone increasing function of $m$. Set $m_0=(1+1/r)g^d_r(k)-1$. Then we have $S^d_r(m_0,k) \geq m_0-g^d_r(k)+1 = g^d_r(k)/r$. Thus, for $m \geq m_0$, we have
\begin{eqnarray}\nonumber S^d_r(m,k) & \geq & S^d_r(m_0,k)\binom{m}{r}/\binom{m_0}{r} \\ & \geq & \label{mm}  \frac{1}{r}g^d_r(k)\binom{m}{r}/\binom{(1+1/r)g^d_r(k)}{r} \\ & \geq & \nonumber \frac{1}{er}g^d_r(k)^{1-r}m^r.
\end{eqnarray}

Let $\alpha^d_r(m,k)$ be the maximum $\alpha$ with the property that for every multifamily $\mathcal{F}$ of at most $k$-element sets with cardinality $m$ and VC-dimension at most $d$, if we uniformly at random choose $r$ sets from $\mathcal{F}$ with replacement, the probability that they have pairwise equal intersections is at least $\alpha$. Thus,  \begin{equation}\label{nextbound} \alpha^d_r(m,k) \geq S^d_r(m,k)/\binom{m}{r}+m^{1-r},\end{equation}
where the first term comes from possibly choosing $r$ different sets (in terms of label, if we view the $m$ not necessarily distinct sets as labeled from $1$ to $m$), and the second term comes from possibly choosing the same set $r$ times.
\smallskip

For $m \geq m_0$, by using (\ref{nextbound}) and then (\ref{mm}), we have \[\alpha^d_r(m,k) \geq S^d_r(m,k)/\binom{m}{r} \geq r!S^d_r(m,k)m^{-r} \geq (r-1)!g^d_r(k)^{1-r}/e.\]
For $m<m_0$, using the trivial bound $S^d_r(m,k) \geq 0$, we have \[\alpha^d_r(m,k) \geq m^{1-r} > m_0^{1-r} = \left((1+1/r)g^d_r(k)-1\right)^{1-r} \geq g^d_r(k)^{1-r}/e.\]
As $\alpha^d_r(k)=\inf_m \alpha^d_r(m,k)$, we have the desired bound $\alpha^d_r(k) \geq g^d_r(k)^{1-r}/e$. \end{proof}

Combining the previous lemma with the Erd\H{o}s-Rado bound $f_r(k) \leq  k! (r-1)^k$, and the inequality $g^d_r(k) \leq g_r(k) = (k-1)f_r(k)+1$, we obtain the following corollary.

\begin{corollary}\label{small}
For any integers $d,k,r \geq 2$, we have
\[\alpha^d_r(k) \geq \left(k!(r-1)^{k+1}+1\right)^{1-r}/e.\]
\end{corollary}

We are now in a position to prove the following result which, together with (\ref{upper}), immediately implies Theorem \ref{main}.

\begin{theorem}\label{prob}

For any $d, k, r\geq 2$, we have

\[\alpha^d_r(k) \geq 2^{-10k(dr)^{2\log^{\ast}k}}.\]

\end{theorem}

\begin{proof}   If $r = 2$, then we have $\alpha_2^d(k) = 1$ and the result follows.  Therefore we can assume $r \geq 3$. We use induction on $k$. For the base cases  $k < 8$, by Corollary \ref{small}, we have

\[\alpha^d_r(k) \geq \left(k!(r-1)^{k+1}+1\right)^{1-r}/e \geq 2^{-10k(dr)^{2\log^{\ast}k}}.\]

For the inductive step, let $k \geq 8$ and assume that the statement holds for all $k' < k$.  Let $\mathcal{F}$ be a non-empty multifamily of at most $k$-element sets with VC-dimension at most $d$.  Without loss of generality, we may assume that the ground set is $\mathbb{N}$. Let $\epsilon_i$ be the fraction of sets in $\mathcal{F}$ that contain $i$. By reordering the elements of the ground set, if necessary, we may also assume that $\epsilon_1 \geq \epsilon_2 \geq \ldots$, that is, the elements of the ground set are ordered in decreasing frequency.
\smallskip

As each member of $\mathcal{F}$ has size at most $k$, the expected size of the intersection of $[s]=\{1,2,\ldots,s\}$ with a randomly selected member of $\mathcal{F}$ is at most $k$. On the other hand, this expectation is $\epsilon_1+\cdots+\epsilon_s \geq s\epsilon_s$. Therefore, we have $\epsilon_s \leq k/s$.

Set $s=\lceil4k^4/\alpha^d_r(\log k)\rceil$.  Define two multifamilies, $\mathcal{F}_1$ and $\mathcal{F}_2$, as follows. Let
\[\mathcal{F}_1=\{S:S \in \mathcal{F}~\textrm{and}~|S \cap [s]| \leq \log k\},\;\;\;\;\; \mathcal{F}_2=\mathcal{F}\setminus \mathcal{F}_1.\] Thus, we have $|\mathcal{F}|=|\mathcal{F}_1|+ |\mathcal{F}_2|$.   We select at random, uniformly and independently with repetition, $r$ sets $S_1,\ldots, S_r \in \mathcal{F}$.  Let $X$ denote the event that the $r$ sets form an $r$-sunflower.  The proof now falls into two cases.

\medskip

\noindent \emph{Case 1:} Suppose that $|\mathcal{F}_1| \geq (1-1/r)|\mathcal{F}|$.  Let $Y$ denote the event that $S_1,\ldots, S_r \in \mathcal{F}_1$. Let $Z$ be the event that $S_1\cap [s],\ldots, S_r\cap [s]$ have pairwise equal intersections, and let $W$ be the event that $S_1\setminus [s],\ldots, S_r\setminus[s]$ are pairwise disjoint.  Hence,

\begin{equation}\label{ineq}
  \mathbb{P}[X]  \geq \mathbb{P}[Y\cap Z\cap W] = \mathbb{P}[Y\cap Z]- \mathbb{P}[Y\cap Z \cap \bar W] \geq \mathbb{P}[Y\cap Z] - \mathbb{P}[\bar{W}].
\end{equation}

\noindent Clearly, we have

\begin{equation}\label{Y}
\mathbb{P}[Y] \geq  (1 - 1/r)^r \geq \frac14,
\end{equation}

\noindent and, by definition,

\begin{equation}\label{ZY}
   \mathbb{P}[Z\mid Y] \geq \alpha^d_r(\log k).
\end{equation}

Therefore, by (\ref{Y}) and (\ref{ZY}), we have

\begin{equation}\label{YcapZ}
\mathbb{P}[Y \cap Z] = \mathbb{P}[Y] \mathbb{P}[Z\mid Y] \geq \frac{1}{4}\alpha^d_r(\log k).
\end{equation}

\noindent Fixing $S_i \setminus [s]$, which has size at most $k$, the probability that $S_j \setminus [s]$ contains at least one of the elements of $S_i \setminus [s]$ is at most $k\epsilon_{s+1} \leq k^2/(s+1)$. Hence, by the probability union bound, we have

 \begin{equation}\label{WC} \mathbb{P}[\bar{W}] \leq \frac{\binom{k}{2}k^2}{s+1}<\frac{k^4}{2s} \leq\frac{\alpha^d_r(\log k)}{8}.\end{equation}

\noindent Combining (\ref{ineq}), (\ref{YcapZ}), and (\ref{WC}), we obtain

\[\mathbb{P}[X]=\mathbb{P}[Y\cap Z] - \mathbb{P}[\bar{W}]  \geq \frac{\alpha^d_r(\log k)}{4} - \frac{\alpha^d_r(\log k)}{8} = \frac{\alpha^d_r(\log k)}{8}.\]

\noindent Hence, by the induction hypothesis, we have
\[\alpha^d_r(k) \geq \mathbb{P}[X] \geq \frac{1}{8}\alpha^d_r(\log k) \geq \frac{1}{8} 2^{-10(\log k)(dr)^{2\log^{\ast}k - 2}} \geq 2^{-10k(dr)^{2\log^{\ast}k}}.\]

 \medskip

\noindent \emph{Case 2}: Suppose that $|\mathcal{F}_2| \geq |\mathcal{F}|/r$. Since $\mathcal{F}$ has VC-dimension at most $d$, by the Sauer-Shelah-Perles lemma, Lemma~\ref{ss}, the number of distinct sets in $\{S \cap [s]:S \in \mathcal{F}\}$ is at most $s^d$.  By the pigeonhole principle, there is a subset $A\subset [s]$ with $|A| \geq \log k$ such that the family \[\mathcal{F}' = \{S \in \mathcal{F}: S\cap [s] = A\}\] has at least $|\mathcal{F}_2|/s^d \geq |\mathcal{F}|/(rs^d)$ members.
\smallskip

Select  $r$ sets $S_1,\ldots,S_r$ from $\mathcal{F}$ uniformly at random with repetition. Let $Y'$ denote the event that $S_1,\ldots, S_r \in \mathcal{F}'$ and let $Z'$ denote the event that $S_1\setminus [s],\ldots, S_r\setminus[s]$ form an $r$-sunflower.  Hence,
\[
\begin{array}{ccl}
\mathbb{P}[X]  & \geq   &   \mathbb{P}[Y'\cap Z'] \\\\
     & = &  \mathbb{P}[Y']\cdot \mathbb{P}[Z'\mid Y']\\\\
     & \geq & \left(\frac{1}{rs^d}\right)^r \alpha^d_r(k-\log k)\\\\
     & \geq &  \frac{1}{r^r(5k^4)^{dr}}\left(\alpha^d_r(\log k)\right)^{dr} \alpha^d_r(k-\log k).
     \end{array}\]

\noindent By the induction hypothesis, we obtain

\[\mathbb{P}[X]  \geq  \frac{1}{r^r(5k^4)^{dr}}\left(2^{-10(\log k)(dr)^{2\log^* k - 2}}\right)^{dr}\left( 2^{-10(k - \log k)(dr)^{2\log^*k}} \right).\]

\noindent Since $dr\geq 6$ and $k \geq 8$, we have

\[\mathbb{P}[X]\geq  \frac{1}{r^r(5k^4)^{dr}}2^{-10k(dr)^{2\log^*k} + 8\log k(dr)^{2\log^*k}} \geq 2^{-10k(dr)^{2\log^{\ast}k}}.\]

\noindent This completes the proof.\end{proof}

\section{Littlestone dimension--Proof of Theorem~\ref{littlestone}}\label{sec4}

Originally, the Littlestone dimension was introduced for the characterization of regret bounds in online learning, see \cite{ALMM,Little,BPS}. As Chase and Freitag~\cite{ChFr} pointed out, the notion is equivalent to Shelah's model theoretic rank. The definition can also be reformulated as follows.

For a finite family $\mathcal{F}$ of sets with ground set $V$, define $\LS(\mathcal{F})$, the {\it Littlestone dimension} of $\mathcal F$, recursively. If $|\mathcal{F}| \leq 1$, then let $\LS(\mathcal{F})=0$. For an element $x$ of the ground set, let $\mathcal{F}_x = \{S \setminus \{x\}: x \in S~\textrm{and}~ S \in \mathcal{F}\}$ and $\mathcal{F}'_x = \{S: x \not \in S~\textrm{and}~ S \in \mathcal{F}\}$. If $|\mathcal{F}|>1$, then let \[\LS(\mathcal{F})=1+\max_{x \in V} \min\left(\LS(\mathcal{F}_x),\LS(\mathcal{F}'_x)\right).\]

For $d\geq 1$, let $h^d_r(k)$ be the minimum positive integer $m$ such that every family of $k$-sets with size at least $m$ and Littlestone dimension at most $d$ contains an $r$-sunflower.

\begin{lemma} For positive integers $k$ and $r$, we have $h^1_r(k)=k+r-1$.
\end{lemma}
\begin{proof}
We have $h^1_r(k) > k+r-2$ by considering the following family $\mathcal{F}_{r,k}$ of $k$-sets. For $k=1$, let the family consist of $r-1$ singleton sets. For $k>1$, we obtain $\mathcal{F}_{r,k}$ from $\mathcal{F}_{r,k-1}$ by adding one new ground element to all sets in $\mathcal{F}_{r,k-1}$, and then including one additional $k$-set with entirely new ground elements. It is straightforward to check that this family of $k$-sets has $k+r-2$ members, its Littlestone dimension is $1$, and it does not contain any $r$-sunflower.
\smallskip

We prove the upper bound inductively on $k$, with the base case $k=1$ being trivial.   Let $k\geq 2$ and let $\mathcal{F}$ be a family of $k$-sets with size $h^1_r(k)-1$ which has Littlestone dimension at most $1$ and does not contain an $r$-sunflower. A family of sets has Littlestone dimension at most $1$ if and only if every element $x$ of the ground set belongs to {\em at most one} or to {\em all but at most one} set in the family, that is, if $|\mathcal{F}_x| \leq 1$ or $|\mathcal{F}'_x| \leq 1$ for all $x$. If there is an element $x$ for which $|\mathcal{F}'_x| \leq 1$, then $|\mathcal{F}_x|=|\mathcal{F}|-1=h^1_r(k)-2$ and $\mathcal{F}_x$ is a family of $(k-1)$-sets of Littlestone dimension at most $1$ which does not contain an $r$-sunflower, from which we obtain $h^1_r(k-1) \leq h^1_r(k-1)+1$. If there is no ground element $x$ in more than one set in $\mathcal{F}$, then all members of $\mathcal{F}$ are disjoint. Therefore, $|\mathcal F|<r$ and $h^1_r(k) \leq r$.
\end{proof}

\begin{lemma}\label{popularelement}
For any family $\mathcal{F}$ of sets of size at most $k$ with no $(r+1)$-sunflower, there is an element of the ground set which belongs to at least a $\frac{1}{kr}$-fraction of the sets.
\end{lemma}
\begin{proof}
Consider a maximum family $\{S_1,\ldots,S_s\}$ of sets in $\mathcal{F}$ which are pairwise disjoint. Such a family forms a sunflower and hence $s \leq r$. In particular, any set in $\mathcal{F}$ contains at least one element from $\bigcup_{i=1}^s S_i$, which has a total of $ks \leq kr$ elements. By the pigeonhole principle, there is an element of the ground set which belongs to at least a fraction $\frac{1}{kr}$ of the sets in $\mathcal{F}$.
\end{proof}

\begin{lemma}\label{recursiveLS}
For integers $k,r \geq 1$ and $d\geq 2$, we have \[h^d_r(k) \leq \max\left(k(r-1)\left(h^{d-1}_r(k-1)-1\right)+1,h^d_r(k-1)+h^{d-1}_r(k)-1\right).\]
\end{lemma}
\begin{proof}
Let $\mathcal{F}$ be a family of $k$-sets with size $h^d_r(k)-1$ which has Littlestone dimension at most $d$ and does not contain an $r$-sunflower. By Lemma \ref{popularelement}, there is an element $x$ of the ground set in at least a fraction $\frac{1}{k(r-1)}$ of the sets in $\mathcal{F}$. As $\mathcal{F}$ has Littlestone dimension $d$, at least one of $\mathcal{F}_x$ or $\mathcal{F}'_x$ has Littlestone dimension at most $d-1$.
\smallskip

If $\mathcal{F}_x$ has Littlestone dimension at most $d-1$, then $\mathcal{F}_x$ is a family of $(k-1)$-sets which has no $r$-sunflower, and hence
\[\frac{1}{k(r-1)}\left(h^d_r(k)-1\right) = \frac{1}{k(r-1)}|\mathcal{F}| \leq |\mathcal{F}_x| \leq h^{d-1}_r(k-1)-1,\] from which it follows that $h^d_r(k) \leq k(r-1)\left(h^{d-1}_r(k-1)-1\right)+1$.
\smallskip

If $\mathcal{F}'_x$ has Littlestone dimension at most $d-1$, then we have $|\mathcal{F}'_x| \leq h^{d-1}_r(k)-1$ and $|\mathcal{F}'_x| \leq h^d_r(k-1)-1$, from which it follows that \[h^d_r(k)-1 = |\mathcal{F}|=|\mathcal{F}_x|+|\mathcal{F}'_x| \leq  h^{d-1}_r(k)-1+h^d_r(k-1)-1,\] and, hence, $h^d_r(k) \leq h^{d-1}_r(k)+h^d_r(k-1)-1$.
\end{proof}

We can now prove Theorem~\ref{littlestone}.

\begin{proof}[Proof of Theorem~\ref{littlestone}]
For the upper bound, the proof is by induction on the Littlestone dimension $d$. In the base case $d = 1$, we have $h^1_r(k) = k+r-1 \leq kr$. Suppose $d\geq 2$.  Consider the recursive upper bound on $h^d_r(k)$ from Lemma \ref{recursiveLS}. We split the proof into two cases depending on the maximum of the two functions in the upper bound on $h_r^d(k)$. In each case, we use the induction hypothesis. 

In the first case, we have \[h^d_r(k) \leq k(r-1)\left(h^{d-1}_r(k-1)-1\right)+1 \leq kr(kr)^{d-1}=(kr)^{d}.\] In the latter case, we have \[h^d_r(k) \leq h^{d-1}_r(k)+h^d_r(k-1)-1 < (kr)^{d-1}+((k-1)r)^{d} \leq (kr)^{d-1}+(1-\frac{1}{k})(kr)^{d} < (kr)^{d}.\] In either case, we obtained the desired bound.

\medskip 

For the lower bound, let $d \geq 6$, $r \geq 3$, $k$ be sufficiently large with $k \geq 4d$, $n=k^2r/(500d\log k)$, $t=\lceil \log d \rceil$ and $m=n^{-1}(n/k)^{d-t}$. We use the probabilistic method to show that there exists a family $\mathcal{F}$ of $k$-element subsets of $[n]:=\{1,\ldots,n\}$ with $|\mathcal{F}| \geq m/2$, $\mathcal{F}$ does not contain any $r$-sunflower, and the Littlestone dimension of $\mathcal{F}$ is at most $d$. This implies the desired lower bound on $h^d_r(k)$. 

We will show that $\mathcal{F}$ satisfies four properties each with high probability. This means that the probability is of the form $1-o(1)$ with the $o(1)$ term tending to $0$ as $k$ tends to infinity. Hence, all four properties hold with high probability. These four properties guarantee that $\mathcal{F}$ has the desired properties and hence there is a choice of $\mathcal{F}$ with the desired properties. 

Pick $m$ subsets $S_1,\ldots,S_m \subset {[n] \choose k}$ uniformly and independently at random. Let $\mathcal{F}$ be the family of distinct $S_i$.  Since $m \ll {n \choose k}$, it is easy to see that, with high probability, we have $|\mathcal{F}| \geq m/2$. 

We next show that, with high probability, $\mathcal{F}$ does not contain any $r$-sunflower. Consider a subsequence of $r$ of these random sets, say $S_1,\ldots,S_r$. The number of sequences of $r$ sets in ${[n] \choose k}$ which have pairwise equal intersection, and this intersection has size $s$, is \[{n \choose s}\prod_{i=1}^r {n-s-(i-1)(k-s) \choose k-s} = s!^{-1}(k-s)!^{-r}n! / \left(n-s-r(k-s)\right)!\]
This is because there are ${n \choose s}$ ways of choosing the common intersection of size $s$, and given the $S_j$ with $j<i$, the remaining $k-s$ elements from $S_i$ not in the common intersection must be chosen from the $n-s-(i-1)(k-s)$ elements not in any of the $S_j$ with $j<i$. As there are ${n \choose k}$ ways to pick each $S_i$, the probability that the $r$ random sets $S_1,\ldots,S_r$ have pairwise equal intersection of size $s$ is 
 \begin{equation}\label{threeproduct} {n \choose k}^{-r} s!^{-1}(k-s)!^{-r}\frac{n!}{\left(n-s-r(k-s)\right)!} = {k \choose s} \cdot \left(k!/(k-s)!\right)^{r-1} \cdot \frac{n!/(n-s-r(k-s))!}{\left(k!{n \choose k}\right)^{r}} .\end{equation}

Note that the expression in the right hand side of (\ref{threeproduct}) is the product of three factors. The middle factor is at most $k^{s(r-1)}$. In the third factor in the right hand side of (\ref{threeproduct}), the numerator can be expressed as the product of factors $(n-j)$ for $j=0,\ldots,s+r(k-s)-1$, which is a total of $s+r(k-s)$ factors, while the denominator can be expressed as the product of $rk$ factors which are of the form $(n-h)$ with $h \leq k$. It follows that the third factor in the right hand side of (\ref{threeproduct}) is at most \[(n-k)^{-s(r-1)}\prod_{j=1}^{(r-1)(k-s)-1}\left(1-\frac{j}{n-k}\right) \leq (n-k)^{-s(r-1)}e^{-(r-1)^2(k-s)^2 /(4n)},\]
where we used the inequality $1-x \leq e^{-x}$ for $x\geq 0$ to bound each factor in the product. 

It follows that the expression on the right hand side of (\ref{threeproduct}) is at most \begin{equation}\label{nextone}\left(k\left(k/(n-k)\right)^{r-1}\right)^s \cdot e^{-(r-1)^2(k-s)^2/(4n)}.\end{equation} Thus, the probability that $S_1,\ldots,S_r$ form an $r$-sunflower is at most \[\label{sumprob}\sum_{s=0}^k \left(k\left(k/(n-k)\right)^{r-1}\right)^se^{-(r-1)^2(k-s)^2/(4n)}.\]
We bound the probability that there is an $r$-sunflower in $\mathcal{F}$ by taking the union bound over all  the ${m \choose r}$ choices of $r$ sets from $S_1,\ldots,S_m$. Note that (\ref{nextone}) is the product of two factors which are each at most $1$. We bound (\ref{nextone}) for $s \geq 2d$ by the first factor, and for $s < 2d$ by the second factor. We also use the inequality ${m \choose r} \leq (em/r)^r$. Substituting in the chosen values for $n$ and $m$, we get a $o(1)$ probability that there is an $r$-sunflower in $\mathcal{F}$. 

Finally, we bound the probability that $\mathcal{F}$ has Littlestone dimension greater than $d$. If $\mathcal{F}$ has Littlestone dimension greater than $d$,  in the rooted complete binary tree realizing the Littlestone dimension, going down from the root by taking the left-child each time for $d-t$ levels, we see that there are at least $2^{t}\geq d$ sets that each contain the same $d-t$ vertices. So the probability that $\mathcal{F}$ has Littlestone dimension greater than $d$ is at most the probability that there are $d$ sets in $\mathcal{F}$ that each contain the same $d-t$ elements from $[n]$. This probability in turn is at most \[{m \choose d}{n \choose d-t}\left(\frac{k}{n}\right)^{(d-t)d}<\left(\frac{em}{d}\right)^d\left(\frac{en}{d-t}\right)^{d-t}\left(\frac{k}{n}\right)^{(d-t)d}=o(1).\] Here we used the union bound over all ${m \choose d}$ choices of $d$ indices from $[m]$ and over all ${n \choose d-t}$ choices of $d-t$ distinct integers in $[n]$. We also used that the probability that a given set of $d-t$ elements in $[n]$ is in a random $k$-set is at  most $\left(k/n\right)^{d-t}$. The last inequality is by substituting in the chosen values of $n$ and $m$. \end{proof}

\section{Geometric versions of the sunflower conjecture}\label{sec5}

We start with the proof of Theorem \ref{disks}. We need the following lemma due to Sharir \cite{Sha}.

\begin{lemma}\label{sharir}
Let $\mathcal{D} = \{D_1,\ldots,D_t\}$ be a family of pseudo-disks in the plane, and let $P$ denote the set of all intersection points of the boundaries of $D_i$. Then the number of points in $P$ covered by the interior of at most $k$ other regions $D_i$ is at most $26kt$.
\end{lemma}

\begin{proof}[Proof of Theorem \ref{disks}]   Given $r > 2$, let $N = (500  + r)^{900k}$.   We proceed by induction on $k$.  The base case $k = 1$ is trivial.  Now assume the statement holds for $k' < k$.
\smallskip

Let $V$ be a planar point set and let $\mathcal{D} = \{D_1,\ldots,D_N\}$ be a family of pseudo-disks in the plane such that $|D_i\cap V| =k$ for all $i$.  By slightly perturbing each region $D_i$, we can assume that no point in $V$ lies on the boundary of $D_i$ for all $i$.  Set $S_i = D_i\cap V$ and $\mathcal{F} = \{S_1,\ldots, S_N\}$.
\smallskip

Let $t = \lambda(\mathcal{F})$ and suppose the sets $S_1,\ldots, S_{t} \in \mathcal{F}$ have the property that for any $1 \leq i < j \leq t$, there is a vertex $v \in S_i\cap S_j$ from $V$ such that $v \not\in S_{\ell}$ for $\ell \in [t]\setminus\{i,j\}$.  Then, by letting $C_i$ denote the boundary of $D_i$, there are at least $\binom{t}{2}$ connected components in $\mathbb{R}^2\setminus \bigcup_i C_i$ that are covered by at most two regions $D_i$.  On the other hand, by Lemma \ref{sharir}, there are at most $4(52)t$ such regions, since every point in the arrangement $\bigcup_i C_i$ is incident to at most four such connected components.  Therefore, we have

\[\binom{t}{2} \leq 208t,\]

\noindent which implies that $t \leq 417.$  

Further, we can assume that $\nu(\mathcal{F}) \leq r-1$, since otherwise we would be done.  By Lemma~\ref{stab}, we have

\[\tau(\mathcal{F}) \leq 11(417)^2(419 + r)\binom{416 + r}{417}^2\leq (500 + r)^{900}.\]

\noindent There is a vertex $v \in V$ which is incident to at least $N/\tau(\mathcal{F})$ members in $\mathcal{F}$. Let $\mathcal{D}' = \{D_i \in \mathcal{D} : v\in D_i\}$, $V' = V\setminus\{v\}$, and $S'_i = V'\cap D'_i$. Hence, $|\mathcal{D}'| \geq N/\tau(\mathcal{F}) \geq  (500  + r)^{900(k-1)}$.  By the induction hypothesis, there are $r$ sets $S'_{i_1},\ldots, S'_{i_r}$ in $\mathcal{D}'$ that form an $r$-sunflower.  Together with vertex $v$, we obtain an $r$-sunflower in $\mathcal{F}$. \end{proof}

Replacing Lemma \ref{sharir} with Clarkson's theorem on levels in arrangement of hyperplanes \cite{Cla}, the argument above gives the following.

\begin{theorem}

Given $r > 2$, there is a constant $C  = C(r)$ for which the following statement is true.  If $V$ is a point set in $\mathbb{R}^3$ and $\mathcal{H} = \{H_1,\ldots,H_N\}$ is a family of $N \geq C^k$ half-spaces such that the size of the set $S_i = H_i\cap V$ is $k$ for all $i$, then there are $r$ distinct sets $S_{i_1},\ldots,S_{i_r}$ that form an $r$-sunflower.
\end{theorem}

\section{Concluding Remarks}
The Erd\H{o}s-Rado sunflower conjecture remains an outstanding open problem. Although we made progress in this paper, it still remains open for families of bounded VC-dimension. 

We were able to prove  the conjecture in a geometric setting in the plane (Theorem \ref{disks}). We think it would be interesting to prove the conjecture in other geometric settings, such as for families of sets that are the intersection of the ground set with semi-algebraic sets of bounded description complexity. Such families are of bounded VC-dimension. The following conjecture is a natural special case to consider.  

\begin{conjecture}
For each integer $r \geq 3$, there is a constant $C=C(r)$ such that the following holds. If $V \subset \mathbb{R}^3$ and $\mathcal{F}$ is a family of subsets of $V$ each of size $k$ with $|\mathcal{F}| \geq C^k$ such that every set in $\mathcal{F}$ is the intersection of $V$ with a unit ball in $\mathbb{R}^3$, then $\mathcal{F}$ contains an $r$-sunflower. 
\end{conjecture}

\medskip

\noindent {\bf Acknowledgement.} We would like to thank Amir Yehudayoff for suggesting working with the Littlestone dimension, and the SoCG 2021 referees for helpful comments.

\end{document}